\documentclass[12pt]{amsart}
\usepackage{}

\usepackage{amsmath}
\usepackage{amsfonts}
\usepackage{amssymb}
\usepackage[all]{xy}           

\usepackage{bbding}
\usepackage{txfonts}
\usepackage{amscd}

\usepackage[shortlabels]{enumitem}
\usepackage{ifpdf}
\ifpdf
  \usepackage[colorlinks,final,backref=page,hyperindex]{hyperref}
\else
  \usepackage[colorlinks,final,backref=page,hyperindex]{hyperref}
\fi
\usepackage{tikz}
\usepackage[active]{srcltx}

\makeatletter

\newtheorem{thm}{Theorem}[section]
\newtheorem{lem}[thm]{Lemma}
\newtheorem{cor}[thm]{Corollary}
\newtheorem{pro}[thm]{Proposition}
\newtheorem{ex}[thm]{Example}
\newtheorem{rmk}[thm]{Remark}
\newtheorem{defi}[thm]{Definition}

\newcommand {\emptycomment}[1]{}

\newcommand{\be }{\begin{equation}}
\newcommand{\ee }{\end{equation}}

\newcommand{\g}{\mathfrak g}
\newcommand{\h}{\mathfrak h}

\newcommand{\huaR}{\mathcal{R}}

\newcommand{\huaC}{{\mathcal{C}}}

\newcommand{\huaO}{{\mathcal{O}}}

\newcommand{\frkg}{\mathfrak g}

\newcommand{\frks}{\mathfrak s}

\newcommand{\Id}{\rm{Id}}

\newcommand{\br}[1]{   [ \cdot,    \cdot  ]   }

\newcommand{\dM}{\mathrm{d}}

\newcommand{\Hom}{\mathrm{Hom}}

\newcommand{\Sym}{\mathrm{Sym}}
\newcommand{\Nat}{\mathbb N}

\newcommand{\gl}{\mathfrak {gl}}

\newcommand{\ad}{\mathrm{ad}}

\newcommand{\K}{\mathbb{K}}

\newcommand{\CE}{\mathrm{CE}}
\begin{document}

\title[Deformations of morphism pre-Lie algebras ]{Cohomologies and deformations of pre-Lie-morphism triples }

\author{Yibo Wang}
\address{School of Mathematics and Statistics, Northeast Normal University,\\
 Changchun 130024, Jilin, China }
\email{wangyb843@nenu.edu.cn}

\author{Shilong Zhang}
\address{College of Science, Northwest A\&F University, Yangling 712100, Shaanxi, China }
\email{shlzhang11@163.com}

\author{Jiefeng Liu$^{\ast}$}
\address{School of Mathematics and Statistics, Northeast Normal University,\\
 Changchun 130024, Jilin, China }
\email{liujf534@nenu.edu.cn}
\thanks{$^{\ast}$ the corresponding author}
\vspace{-5mm}


\begin{abstract}
A (pre-)Lie-morphism triple consists of two (pre-)Lie algebras and a (pre-)Lie algebra homomorphism between them.  We give chomologies of  pre-Lie-morphism triples. As an application, we study the infinitesimal deformations of pre-Lie-morphism triples. Finally, we show that the cohomology of the pre-Lie-morphism triple can be deduced from a new cohomology of the Lie-morphism triple.
 \end{abstract}




\maketitle

\tableofcontents

\allowdisplaybreaks

\section{Introduction}

Pre-Lie algebras are a class of nonassociative algebras coming from the
study of convex homogeneous cones, affine manifolds and affine
structures on Lie groups, and the aforementioned cohomologies of
associative algebras.  They also appeared in many fields of
mathematics and mathematical physics, such as
symplectic structures on Lie groups and Lie algebras, integrable
systems, Poisson brackets and infinite dimensional Lie algebras,
vertex algebras, quantum field theory, operads and numerical analysis. See the survey
\cite{Pre-lie algebra in geometry,Man} and the references therein for
more details.

The cohomology theory is a classical approach  associating
invariants to a mathematical structure.  Cohomology controls
deformations and extension problems of the corresponding algebraic
structures. Cohomology theories of various kinds of algebras have
been developed and studied in \cite{Ch-Ei,Ge0,Har,Hor}. See the
review paper \cite{GLST} for more details. The cohomologies of Lie algebra homomorphisms were introduced by Nijenhuis and Richardson in \cite{ NR1}.  Then the cohomology of two Lie algebras with a Lie algebra homomorphism between them  were studied by Y. Fr\'egier in \cite{Fregier}, which can control simultaneous deformations of these two Lie algebras with the Lie algebra homomorphism. We call two Lie algebras with a Lie algebra homomorphism between them a Lie-morphism triple.  See \cite{Das} for more details on the cohomology of a Lie-morphism triple with coefficients in a representation and its  applications.

In this paper, following the idea  given by Y. Fr\'egier in \cite{Fregier}, we study the cohomologies of pre-Lie-morphism triples. Here  a pre-Lie-morphism triple consists of two pre-Lie algebras and a pre-Lie algebra homomorphism between them. Pre-Lie-morphism triples arise from Rota-Baxter operators, Nijenhuis operators, symplectic Lie algebra homomorphisms, $\mathfrak s$-matrices, Cayley maps and $B$-series. As an application of the cohomology of pre-Lie-morphism triples, we study infinitesimal deformations of  pre-Lie-morphism triples and show that equivalent pre-Lie-morphism triples are in the same cohomology class. We also introduce  the notion of a Nijenhuis pair on a pre-Lie-morphism triple and show that it can generate a trivial infinitesimal deformation. We show that compatible $\huaO$-operators on pre-Lie algebras can give rise to Nijenhuis pairs naturally. See \cite{Dorfman1993,WBLS} for more details about Nijenhuis operators on pre-Lie algebras and Lie algebras. Based on the close relationship between the cohomologies of pre-Lie algebras and the cohomologies of Lie algebras, we find that the cohomology of a pre-Lie-morphism triple  can be deduced from a new cohomology of the Lie-morphism triple.

The paper is organized as follows. In Section \ref{sec:Pre}, we recall representations, cohomologies and homomorphisms of pre-Lie algebras. Then some examples of pre-Lie algebra homomorphisms are given. In Section \ref{sec:main}, we give the cohomology of a pre-Lie-morphism triple and use it to study infinitesimal deformations of the pre-Lie-morphism triple. Furthermore,  we introduce the notion of a Nijenhuis pair  on the pre-Lie-morphism triple , which generates trivial infinitesimal deformations.  At last, we construct Nijenhuis pairs  from compatible $\huaO$-operators on pre-Lie algebras.  In Section \ref{sec:relation}, we study the relationship between the cohomologies of pre-Lie-morphism triples and the cohomologies of Lie-morphism triples.

 In this paper, we work over an algebraically closed field $\K$ of characteristic 0 and all the vector spaces are taken over $\K$. Unless
otherwise noted, all vector spaces we consider will be finite dimensional .
\section{Preliminaries}\label{sec:Pre}
In this section, we recall the notion of  pre-Lie algebra homomorphisms and give some examples.
\begin{defi}  A {\bf pre-Lie algebra} is a pair $(\g,\cdot_\g)$, where $\g$ is a vector space and  $\cdot_\g:\g\otimes \g\longrightarrow \g$ is a bilinear multiplication
satisfying that for all $x,y,z\in \g$, the associator
$(x,y,z)=(x\cdot_\g y)\cdot_\g z-x\cdot_\g(y\cdot_\g z)$ is symmetric in $x,y$,
i.e.
$$(x,y,z)=(y,x,z),\;\;{\rm or}\;\;{\rm
equivalently,}\;\;(x\cdot_\g y)\cdot_\g z-x\cdot_\g(y\cdot_\g z)=(y\cdot_\g x)\cdot_\g
z-y\cdot_\g(x\cdot_\g z).$$
\end{defi}

Let $(\g,\cdot_\g)$ be a pre-Lie algebra. The commutator $
[x,y]_\g=x\cdot_\g y-y\cdot_\g x$ defines a Lie algebra structure
on $\g$, which is called the {\bf sub-adjacent Lie algebra} of
$(\g,\cdot_\g)$ and denoted by $\g^c$. Furthermore,
$L:\g\longrightarrow \gl(\g)$ with $x\rightarrow L_x$, where
$L_xy=x\cdot_\g y$, for all $x,y\in \g$, gives a representation of
the Lie algebra $\g^c$ on $\g$. See \cite{Pre-lie algebra in
geometry} for more details.

\begin{defi}
Let $(\g,\cdot_\g)$ be a pre-Lie algebra and $V$  a vector
space. A {\bf representation} of $\g$ on $V$ consists of a pair
$(\rho,\mu)$, where $\rho:\g\longrightarrow \gl(V)$ is a representation
of the Lie algebra $\g^c$ on $V $ and $\mu:\g\longrightarrow \gl(V)$ is a linear
map satisfying \begin{eqnarray}\label{representation condition 2}
 \rho(x)\mu(y)u-\mu(y)\rho(x)u=\mu(x\cdot_\g y)u-\mu(y)\mu(x)u, \quad \forall~x,y\in \g,~ u\in V.
\end{eqnarray}
\end{defi}

Usually, we denote a representation by $(V;\rho,\mu)$. It is
obvious that $(  \K  ;\rho=0,\mu=0)$ is a representation, which is
called the {\bf trivial representation}. Let $R:\g\rightarrow
\gl(\g)$ be a linear map with $x\longrightarrow R_x$, where the
linear map $R_x:\g\longrightarrow\g$  is defined by
$R_x(y)=y\cdot_\g x,$ for all $x, y\in \g$. Then
$(\g;\rho=L,\mu=R)$ is also a representation, which is called the
{\bf regular representation}. Define two linear maps $L^*,R^*:\g\longrightarrow
\gl(\g^*)$   with $x\longrightarrow L^*_x$ and
$x\longrightarrow R^*_x$ respectively (for all $x\in \g$)
by
\begin{equation}
\langle L_x^*(\xi),y\rangle=-\langle \xi, x\cdot y\rangle, \;\;
\langle R_x^*(\xi),y\rangle=-\langle \xi, y\cdot x\rangle, \;\;
\forall x, y\in \g, \xi\in \g^*.
\end{equation}
Then $(\g^*;\rho={\rm ad}^*=L^*-R^*, \mu=-R^*)$ is a
representation of $(\g,\cdot_\g)$, which is called the
{\bf coregular representation}. In fact, it is the dual
representation of the regular representation $(\g;L,R)$.

The cohomology complex for a pre-Lie algebra $(\g,\cdot_\g)$ with coefficients in a representation  $(V;\rho,\mu)$ is given as follows (\cite{Pre-lie algebra in geometry}). The cohomology of the pre-Lie algebra $\g$ with coefficients in $V$ is the cohomology of the cochain complex $(\otimes_{n\geq 1}C_{\rm preLie}^n(\g,V),\dM)$, where $C^n_{\rm preLie}(\g,V)=\Hom(\wedge^{n-1}\g\otimes \g,V)$ and  the coboundary operator $\dM:C_{\rm preLie}^n(\g,V)\longrightarrow C_{\rm preLie}^{n+1}(\g,V)$ is given by
\begin{eqnarray}\label{eq:pre-Lie cohomology}
\dM f(x_1, \cdots,x_{n+1})
\nonumber&=&\sum_{i=1}^{n}(-1)^{i+1}\rho(x_i)f(x_1, \cdots,\hat{x_i},\cdots,x_{n+1})\\
\nonumber &&+\sum_{i=1}^{n}(-1)^{i+1}\mu(x_{n+1})f(x_1, \cdots,\hat{x_i},\cdots,x_n,x_i)\\
\nonumber&&-\sum_{i=1}^{n}(-1)^{i+1}f(x_1, \cdots,\hat{x_i},\cdots,x_n,x_i\cdot_\g x_{n+1})\\
\label{eq:cobold} &&+\sum_{1\leq i<j\leq n}(-1)^{i+j}f([x_i,x_j]_\g,x_1,\cdots,\hat{x_i},\cdots,\hat{x_j},\cdots,x_{n+1}),
\end{eqnarray}
for all $x_i\in \g,~i=1,\cdots,n+1$. We denote the corresponding $n$-th cohomology group by $H_{\rm preLie}^k(\g,V)$.

\begin{defi}
Let $(\g,\cdot_\g)$ and $(\h,\cdot_\h)$  be two  pre-Lie algebras. A {\bf homomorphism} between $\g$ and $\h$ is a linear map $\phi:\g\to\h$ satisfying
\begin{equation}
\phi(x\cdot_\g y)=\phi(x)\cdot_\h \phi(y),\quad\forall~x,y\in\g.
\end{equation}
Moreover, two pre-Lie algebras $\g$ and $\h$ with a pre-Lie algebra homomorphism $\phi$ is call a {\bf pre-Lie-morphism triple}. We denote it by $((\g,\cdot_\g),(\h,\cdot_\h),\phi)$, or simply $(\g,\h,\phi)$.
\end{defi}
It is obvious that if $\phi$ is a homomorphism from a pre-Lie algebra $\g$ to a pre-Lie algebra $\h$, then $\phi$ is a homomorphism from the  sub-adjacent Lie algebra $\g^c$ to the sub-adjacent Lie algebra $\h^c$. We call the Lie-morphism triple $(\g^c,\h^c,\phi)$ a {\bf sub-adjacent Lie-morphism triple} of the pre-Lie-morphism triple $(\g,\h,\phi)$.

Let $D$ be a derivation on a commutative associative algebra $(A,\cdot_A)$. Then the new product
\begin{equation}\label{eq:AssDer1}
x\ast_D y=x\cdot_A D(y),\quad \forall~x,y\in A
\end{equation}
makes $(A,\ast_D)$ being a pre-Lie algebra by  Gel'fand
\rm\cite{GelDor}.
\begin{ex}{\rm
 Let $D_1$ be a derivation on a commutative associative algebra $(A,\cdot_A)$ and $D_2$ a derivation on a commutative associative algebra $(B,\cdot_B)$. Assume that $f:A\to B$ is an associative algebra homomorphism satisfying $f\circ D_1=D_2\circ f$. Then $f$ is a pre-Lie algebra homomorphism  from the pre-Lie algebra $(A,\ast_{D_1})$ to the pre-Lie algebra $(B,\ast_{D_2})$.}
\end{ex}

Let  $(\g,[\cdot,\cdot]_\g)$ be a Lie algebra. Recall that $\omega\in\wedge^2\g^*$ is a 2-cocycle on $\g$ if
\begin{equation}\label{eq:Lie2-form}
\omega([x,y]_\g,z)+\omega([z,x]_\g,y)+\omega([y,z]_\g,x)=0,\quad \forall~x,y,z\in\g.
\end{equation}

A {\bf symplectic  Lie algebra}, denoted by $(\g,[\cdot,\cdot]_\g,\omega)$, is a Lie algebra  $(\g,[\cdot,\cdot]_\g)$ together with a nondegenerate
  $2$-cocycle $\omega\in\wedge^2\g^*$ on $\g$.

\begin{lem}{\rm(\cite{Chu})}\label{thm:sp}
Let $(\g,[\cdot,\cdot]_\g,\omega)$ be a symplectic Lie algebra.
Then there exists a pre-Lie algebra structure ``$\cdot_\g$'' on
$\g$ given by
\begin{equation}\label{eq:LietoLSA}
\omega(x\cdot_\g y,z)=-\omega(y,[x,z]_\g),\quad \forall ~
x,y,z\in\g,
\end{equation}
such that the sub-adjacent Lie algebra is exactly
$(\g,[\cdot,\cdot]_\g)$ itself.
\end{lem}
\begin{ex}{\rm
 Let $(\g,[\cdot,\cdot]_\g,\omega_1)$ and $(\h,[\cdot,\cdot]_\h,\omega_2)$ be two symplectic Lie algebras. Assume that $\phi:\g\to \h$ is a symplectic Lie algebra homomorphism, i.e.
 $$\phi([x,y]_\g)=[\phi(x),\phi(y)]_\h,\quad \omega_2(\phi(x),\phi(y))=\omega_1(x,y),\quad\forall~x,y\in\g.$$
 Then $\phi$ is a pre-Lie algebra homomorphism from the pre-Lie algebra $(\g,\cdot_{\g})$ to the pre-Lie algebra $(\h,\cdot_{\h})$.}

\end{ex}

\begin{defi}{\rm(\cite{AB})}
Let $(\g,\cdot_\g)$ be a pre-Lie algebra and $\huaR:\g\longrightarrow
\g$ a   linear operator. If $\huaR$ satisfies \begin{equation}
\huaR(x)\cdot_\g \huaR(y)=\huaR(\huaR(x)\cdot_\g y+x\cdot_\g \huaR(y)+ \lambda x\cdot_\g y),\quad \forall x,y\in \g,\end{equation} then $\huaR$ is called a
{\bf Rota-Baxter operator of weight $\lambda$} on $\g$.
\end{defi}

\begin{ex}{\rm
  Let $\huaR$ be a Rota-Baxter operator of weight $\lambda$ on a pre-Lie algebra $(\g,\cdot_\g)$. Define a new operation $\cdot^\huaR:\g\otimes \g\longrightarrow \g$ by
  $$x\cdot^\huaR y=\huaR(x)\cdot_\g y+x\cdot_\g \huaR(y)+ \lambda x\cdot_\g y,\;\forall x,y\in\g.$$
  Then $(\g,\cdot^\huaR)$ is a pre-Lie algebra and $\huaR$ is a pre-Lie algebra homomorphism from the pre-Lie algebra $(\g,\cdot^\huaR)$ to the pre-Lie algebra $(\g,\cdot_\g)$.}
\end{ex}

\begin{defi}{\rm(\cite{WBLS})}
Let $(\g,\cdot_\g)$ be a pre-Lie algebra. A linear operator $N:\g\longrightarrow \g$ is called a {\bf Nijenhuis operator} if
\begin{equation}\label{eq:Nijenhuis operator}
N(x)\cdot_\g N(y)=N\big(N(x)\cdot_\g y+x\cdot_\g N(y)-N(x\cdot_\g
y)\big),\;\forall x,y\in\g.
\end{equation}
\end{defi}
\begin{ex}{\rm
  Let $(\g,\cdot_\g)$ be a pre-Lie algebra and $N$ a Nijenhuis operator on $\g$. Define a new operation $\cdot_N:\otimes^2\g\to\g$ by
  \begin{eqnarray}
x \cdot_N y =N(x) \cdot_\g y +x \cdot_\g N(y)-N(x \cdot_\g y),\;\forall x,y\in\g.
\end{eqnarray}
Then $(\g,\cdot_N)$ is a pre-Lie algebra and $N$ is a pre-Lie algebra homomorphism from $(\g,\cdot_N)$ to $(\g,\cdot_\g)$.}
\end{ex}

\begin{defi}{\rm(\cite{BaiO})}
Let $(\g,\cdot_\g)$ be a pre-Lie algebra and $(V;\rho,\mu)$ a representation. A linear map $T:V\longrightarrow \g$ is called  an {\bf $\huaO$-operator} on a pre-Lie algebra
$\g$ associated to a representation $(V;\rho,\mu)$ if it satisfies
\begin{equation}
  T(u)\cdot_\g T(v)=T\Big(\rho(T(u))(v)+\mu(T(v))(u)\Big),\quad \forall u,v\in V.
\end{equation}
\end{defi}
Note that a Rota-Baxter operator of weight zero on a
pre-Lie algebra $\g$ is exactly an $\huaO$-operator
associated to the regular representation $(\g;L,R)$.

\begin{ex}\label{ex:O-operators}
{\rm
Let $T$ be an $\huaO$-operator on a pre-Lie algebra
$(\g,\cdot_\g)$ associated to a representation $(V;\rho,\mu)$. Define a new operation $\cdot^T:\otimes^2\g\to\g$ by
  \begin{eqnarray}
x\cdot^Ty=\rho(T(u))(v)+\mu(T(v))(u),\;\forall x,y\in\g.
\end{eqnarray}
Then $(\g,\cdot^T)$ is a pre-Lie algebra and $T$ is a pre-Lie algebra homomorphism from $(\g,\cdot^T)$ to $(\g,\cdot_\g)$.}
\end{ex}

Let $(\g,\cdot_\g)$ be a pre-Lie algebra and $r\in\Sym^2(\g)$. We introduce $\llbracket r,r\rrbracket\in\wedge^2\g\otimes\g$ as follows
\begin{equation}\label{S-equation1}
\llbracket r,r\rrbracket(\xi,\eta,\zeta)=-\langle\xi,r^\sharp(\eta)\cdot_\g r^\sharp(\zeta)\rangle+\langle\eta,r^\sharp(\xi)\cdot_\g r^\sharp(\zeta)\rangle+\langle\zeta,[r^\sharp(\xi),r^\sharp(\eta)]^c\rangle,\quad\forall~\xi,\eta,\zeta\in\g^*,
\end{equation}
where $r^\sharp:\g^*\to \g$ is given by
$\langle r^\sharp(\xi),\eta\rangle=r(\xi,\eta),$ for all $~\xi,\eta\in\g^*$.

\begin{defi}{\rm(\cite{Left-symmetric bialgebras})} Let
$(\g,\cdot_\g)$ be a pre-Lie algebra.  If $r\in\Sym^2(\g)$ and
satisfies $\llbracket r,r\rrbracket=0$, then $r$ is called an {\bf
$\frks$-matrix} in $(\g,\cdot_\g)$.
\end{defi}

\begin{rmk}
An $\frks$-matrix in a pre-Lie algebra $(\g, \cdot_\g)$ is a
solution of the $S$-equation in $(\g, \cdot_\g)$ which is an
analogue of the classical Yang-Baxter equation in a Lie algebra.
It plays an important role in the theory of pre-Lie bialgebras,
see \cite{Left-symmetric bialgebras} for more details.
\end{rmk}

It was shown in \cite{Left-symmetric bialgebras} that $r\in\Sym^2(\g)$ is an $\frks$-matrix in the pre-Lie algebra $\g$ if and only if  $r^\sharp:\g^*\to \g$  is an $\huaO$-operator
associated to the coregular representation $(\g^*;{\rm ad}^*,-R^*)$.
\begin{ex}\label{ex:s-matric}
 {\rm Let $(\frkg,\cdot_\frkg)$ be a pre-Lie algebra and $r$ an $\frks$-matrix. Then $(\g^*,\cdot_r)$ is a pre-Lie algebra, where the multiplication $\cdot^r:\otimes ^2\g^*\longrightarrow\g^*$ is  given by
\begin{equation}\label{eq:pre-bia}
\xi\cdot^{r}\eta=\ad^*_{r^\sharp(\xi)}\eta-R^*_{r^\sharp(\eta)}\xi,\quad \forall~\xi,\eta\in\g^*.
\end{equation}
Furthermore, $r^\sharp:\g^*\longrightarrow \g$ is a  pre-Lie algebra homomorphism from $(\g^*,\cdot^r)$ to $(\g,\cdot_\g)$.}
\end{ex}

\section{Cohomologies and deformations of pre-Lie-morphism triples}\label{sec:main}
In this section, we give the cohomlogy of the pre-Lie-morphism triple and then use it to study the deformations of  the pre-Lie-morphism triple. In particular, we show that a Nijenhuis pair on the pre-Lie-morphism triple generates a trivial deformation. The relations between Nijenhuis pairs and compatible $\huaO$-operator are discussed.
\subsection{Cohomologies  of  pre-Lie-morphism triples}
Let $\phi$ be a homomorphism  between pre-Lie algebras $(\g,\cdot_\g)$ and $(\h,\cdot_\h)$. Define $\rho_\phi:\g\to \gl(\h)$ and $\mu_\phi:\g\to \gl(\h) $ by
\begin{equation}\label{eq:rep}
\rho_\phi(x)u=\phi(x)\cdot_\h u,\quad \mu_\phi(x)u=u\cdot_\h \phi(x),\quad \forall~x\in\g,u\in\h.
\end{equation}
Furthermore, we have
\begin{pro}
The pair $(\rho_\phi,\mu_\phi)$  is a representation of the pre-Lie algebra $\g$ on $\h$.	
\end{pro}
\begin{proof}
Since $\phi$ is a pre-Lie algebra homomorphism from $\g$ to $\h$, $\phi$ is a Lie algebra homomorphism from sub-adjacent Lie algebra $\g^c$ to $\h^c$. Then we have
\begin{eqnarray*}
\rho_\phi([x,y]_\g)u&=&	\phi([x,y]_\g)\cdot_\h u\\
&=&(\phi(x)\cdot_\h\phi(y))\cdot_\h u-(\phi(y)\cdot_\h\phi(x))\cdot_\h u\\
&=&\phi(x)\cdot_\h(\rho_\phi(y)\cdot_\h u)-\phi(y)\cdot_\h(\rho_\phi(x)\cdot_\h u)\\
&=&\rho_\phi(x)\rho_\phi(y)u-\rho_\phi(y)\rho_\phi(x)u,
\end{eqnarray*}
which implies that $\rho_\phi([x,y]_\g)=\rho_\phi(x)\rho_\phi(y)-\rho_\phi(y)\rho_\phi(x)$.

Furthermore, since $\h$ is a pre-Lie algebra, we have
\begin{eqnarray*}
  &&\rho_\phi(x)\mu_\phi(y)u-\mu_\phi(y)\rho_\phi(x)u-\mu_\phi(x\cdot_\g y)u+\mu_\phi(y)\mu_\phi(x)u\\
  &=&\phi(x)\cdot_\h(u\cdot_\h\phi(y))-(\phi(x)\cdot_\h u)\cdot_\h\phi(y)-u\cdot_\h(\phi(x)\cdot_\h\phi(y))+(u\cdot_\h\phi(x))\cdot_\h\phi(y)\\
  &=&0,
\end{eqnarray*}
which implies that $\rho_\phi(x)\mu_\phi(y)z-\mu_\phi(y)\rho_\phi(x)z=\mu_\phi(x\cdot_\g y)z-\mu_\phi(y)\mu_\phi(x)z$.

Thus $(\rho_\phi,\mu_\phi)$  is a representation of the pre-Lie algebra $\g$ on $\h$
\end{proof}	

Now we give the cohomology complex for pre-Lie-morphism triples. Let $\phi$ be  pre-Lie algebra homomorphism between pre-Lie algebras $(\g,\cdot_\g)$ and $(\h,\cdot_\h)$.  Define the set of $k$-cochains by
$$\huaC_{\rm preLie}^k(\phi,\phi)=C_{\rm preLie}^{k+1}(\g,\g)\oplus C_{\rm preLie}^{k+1}(\h,\h)\oplus C_{\rm preLie}^{k}(\g,\h),k\geq0,$$
in which we set $C_{\rm preLie}^{0}(\g,\h)=0$.

The coboundary operator $\delta:\huaC_{\rm preLie}^k(\phi,\phi)\to \huaC_{\rm preLie}^{k+1}(\phi,\phi)$ is defined  by
$$\delta_{\rm preLie}(f_1,f_2,f_3)=(\dM_\g f_1,\dM_\h f_2,\dM_\phi f_3+(-1)^k(\phi\circ f_1-\phi^*\circ f_2)),$$
where $\dM_\g$ and $\dM_\h$ are the coboundary operators of the pre-Lie algebra $\g$ and $\h$ associated to their regular representations respectively,  $\dM_\phi$ is the coboundary operator of the pre-Lie algebra $\g$ associated to the representation $(\h;\rho_\phi,\mu_\phi)$, and $f_1\in C_{\rm preLie}^{k+1}(\g,\g),f_2\in C^{k+1}(\h,\h),f_3\in C^{k}(\g,\h)$. Here $\phi\circ f_1\in C_{\rm preLie}^{k+1}(\g,\g)$ and $\phi^*\circ f_2\in C^{k+1}(\h,\h)$ are defined by
\begin{eqnarray*}
(\phi\circ f_1)(x_1,\cdots,x_{k+1})&=&\phi\big( f_1(x_1,\cdots,x_{k+1})\big),\\
(\phi^*\circ f_2)(x_1,\cdots,x_{k+1})&=&f_2 (\phi(x_1),\cdots,\phi(x_{k+1})),
\end{eqnarray*}
where $x_1,\cdots,x_{k+1}\in\g$.
\begin{pro}
	We have $\delta_{\rm preLie}\circ\delta_{\rm preLie}=0$. Hence $(\huaC^*(\phi,\phi)=\otimes_{k\geq 1}C^n(\phi,\phi),\delta_{\rm preLie})$ is a cochain complex.
\end{pro}	
\begin{proof}
For $f_1\in C_{\rm preLie}^{k+1}(\g,\g)$, we have
\begin{eqnarray*}
\label{eq:pre-Lie cohomology}
\dM_\phi(\phi\circ f_1) (x_1, \cdots,x_{k+2})
\nonumber&=&\sum_{i=1}^{k+1}(-1)^{i+1}\phi(x_i)\cdot_\h \phi(f_1(x_1, \cdots,\hat{x_i},\cdots,x_{k+2}))\\
\nonumber &&+\sum_{i=1}^{k+1}(-1)^{i+1}\phi(f_1(x_1, \cdots,\hat{x_i},\cdots,x_{k+1},x_i))\cdot_\h\phi(x_{k+2})\\
\nonumber&&-\sum_{i=1}^{k+1}(-1)^{i+1}\phi(f_1(x_1, \cdots,\hat{x_i},\cdots,x_{k+1},x_i\cdot_\g x_{k+2}))\\
\label{eq:cobold} &&+\sum_{1\leq i<j\leq {k+1}}(-1)^{i+j}\phi(f_1([x_i,x_j]_\g,x_1,\cdots,\hat{x_i},\cdots,\hat{x_j},\cdots,x_{k+2}))\\
\nonumber&=&\phi(\sum_{i=1}^{k+1}(-1)^{i+1} x_i\cdot_\g f_1(x_1, \cdots,\hat{x_i},\cdots,x_{k+2}))\\
\nonumber &&+\phi(\sum_{i=1}^{k+1}(-1)^{i+1}f_1(x_1, \cdots,\hat{x_i},\cdots,x_{k+1},x_i))\cdot_\g x_{k+2})\\
\nonumber&&-\phi(\sum_{i=1}^{k+1}(-1)^{i+1}f_1(x_1, \cdots,\hat{x_i},\cdots,x_{k+1},x_i\cdot_\g x_{k+2}))\\
\label{eq:cobold} &&+\phi(\sum_{1\leq i<j\leq {k+1}}(-1)^{i+j}f_1([x_i,x_j]_\g,x_1,\cdots,\hat{x_i},\cdots,\hat{x_j},\cdots,x_{k+2}))\\
\nonumber&=&\phi(\dM_\g(f_1))(x_1, \cdots,x_{k+2}),
\end{eqnarray*}
which implies that
\begin{eqnarray}\label{eq:pre1}
\dM_\phi(\phi(f_1))=\phi(\dM_\g(f_1)).
\end{eqnarray}
For $f_2\in C_{\rm preLie}^{k+1}(\h,\h)$, we have
\begin{eqnarray*}
\label{eq:pre-Lie cohomology}
\dM_\phi(\phi^* f_2) (x_1, \cdots,x_{k+2})
\nonumber&=&\sum_{i=1}^{k+1}(-1)^{i+1}\phi(x_i)\cdot_\h f_2(\phi(x_1, \cdots,\hat{x_i},\cdots,x_{k+2}))\\
\nonumber &&+\sum_{i=1}^{k+1}(-1)^{i+1}f_2(\phi(x_1, \cdots,\hat{x_i},\cdots,x_{k+1},x_i))\cdot_\h\phi(x_{k+2})\\
\nonumber&&-\sum_{i=1}^{k+1}(-1)^{i+1}f_2(\phi(x_1, \cdots,\hat{x_i},\cdots,x_{k+1},x_i\cdot_\g x_{k+2}))\\
\label{eq:cobold} &&+\sum_{1\leq i<j\leq {k+1}}(-1)^{i+j}f_2(\phi([x_i,x_j]_\g,x_1,\cdots,\hat{x_i},\cdots,\hat{x_j},\cdots,x_{k+2}))\\
\nonumber&=&\sum_{i=1}^{k+1}(-1)^{i+1}\phi(x_i)\cdot_\h f_2(\phi(x_1, \cdots,\hat{x_i},\cdots,x_{k+2}))\\
\nonumber &&+\sum_{i=1}^{k+1}(-1)^{i+1}f_2(\phi(x_1, \cdots,\hat{x_i},\cdots,x_{k+1},x_i))\cdot_\h\phi(x_{k+2})\\
\nonumber&&-\sum_{i=1}^{k+1}(-1)^{i+1}f_2(\phi(x_1), \cdots,\hat{x_i},\cdots,\phi(x_{k+1}),\phi(x_i)\cdot_\h \phi(x_{k+2}))\\
\label{eq:cobold} &&+\sum_{1\leq i<j\leq {k+1}}(-1)^{i+j}f_2([\phi(x_i),\phi(x_j)]_\h,\phi(x_1),\cdots,\hat{x_i},\cdots,\hat{x_j},\cdots,\phi(x_{k+2}))\\
\nonumber&=&\phi^*\dM_\h(f_2)(x_1, \cdots,x_{k+2}),
\end{eqnarray*}
which implies that
\begin{eqnarray}\label{eq:pre2}
\dM_\phi(\phi^* f_2)=\phi^*\dM_\h(f_2).
\end{eqnarray}

Furthermore, by \eqref{eq:pre1}, \eqref{eq:pre2} and the fact that $\dM_\g\circ \dM_\g=0$ , $\dM_\h\circ \dM_\h=0$ and $\dM_\phi\circ \dM_\phi=0$, we have
\begin{eqnarray*}
&&\delta_{\rm preLie}\circ\delta_{\rm preLie}(f_1,f_2,f_3)\\
&=&\delta_{\rm preLie}(\dM_\g f_1,\dM_\h f_2,\dM_\phi f_3+(-1)^k(\phi\circ f_1-\phi^*\circ f_2))\\
&=&(\dM_\g\dM_\g f_1,\dM_\h\dM_\h f_2,\dM_\phi(\dM_\phi f_3+(-1)^k(\phi\circ f_1-\phi^*\circ f_2))+(-1)^{k+1}(\phi\circ\dM_\g f_1-\phi^*\circ\dM_\h f_2))\\
&=&\Big(0,0,(-1)^k\big(\dM_\phi(\phi(f_1))-\phi(\dM_\g(f_1)+\phi^*\dM_\h(f_2)-\dM_\phi(\phi^* f_2)\big)\Big)=(0,0,0),
\end{eqnarray*}
which implies that  $\delta_{\rm preLie}\circ\delta_{\rm preLie}=0$.
	\end{proof}	
	
\begin{defi}
The cohomology of the cochain complex $(\huaC^*(\phi,\phi),\delta_{\rm preLie})$ is called the {\bf cohomology of the pre-Lie-morphism triple $(\g,\h,\phi)$}. The corresponding $k$-th cohomology group, which we denote by $H^k(\phi,\phi)$, is called the {\bf $k$-th cohomology group for the pre-Lie-morphism triple $(\g,\h,\phi)$}.	
\end{defi}

\subsection{Infinitesimal deformations of pre-Lie-morphism triples}
Let $\phi$ be a  pre-Lie algebra homomorphism between pre-Lie algebras $(\g,\cdot_\g)$ and $(\h,\cdot_\h)$, $\omega\in \Hom(\otimes^2\g,\g)$, $\varpi\in\Hom(\otimes^2\h,\h)$ and $\theta\in\Hom(\g,\h)$. Consider a  $t$-parameterized family
of multiplications $\cdot_t:\g\otimes \g\longrightarrow \g$, $\ast_t:\g\otimes \g\longrightarrow \g$ and  linear maps $\phi_t:\g\longrightarrow \h$ given by
\begin{eqnarray*}
x\cdot_t y&=&x\cdot_\g y+t\omega(x,y),\\
u\ast_t v&=&u\cdot_\h v+t\varpi(u,v),\\
	\phi_t(x)&=&\phi(x)+t\theta(x),
\end{eqnarray*}
where $x,y\in\g$ and $u,v\in\h$. If  $\g_t=(\g,\cdot_t)$ and $\h_t=(\h,\ast_t)$ are pre-Lie algebras  and
$$\phi_t:\g_t\to\h_t ~\mbox{module}~ t^3$$
 are  pre-Lie algebra  homomorphisms for all $t$, we say that $(\omega,\varphi,\theta)$ generates a {\bf $1$-parameter infinitesimal
	deformation} of $(\g,\h,\phi)$. We denote a $1$-parameter infinitesimal
	deformation of $(\g,\h,\phi)$ by $(\g_t,\h_t,\phi_t)$.

It is straightforward to check that $(\omega,\varpi,\theta)$ generates a $1$-parameter infinitesimal
	deformation of $(\g,\h,\phi)$  if and only if
 \begin{eqnarray}\label{2-cocycle g}
0&=&x\cdot_\g\omega(y,z)-y\cdot_\g\omega(x,z)+\omega(y,x)\cdot_\g z-\omega(x,y)\cdot_\g z  \\
\nonumber&&-\omega(y,x\cdot_\g z)+\omega(x,y\cdot_\g z)+\omega([x,y]_\g,z),\\
\label{omega bracket}0&=&\omega(\omega(x,y),z)-\omega(x,\omega(y,z))-\omega(\omega(y,x),z)+\omega(y,\omega(x,z)),\\
\label{eq:2-cocycle h}0&=&u\cdot_\h\varpi(v,w)-v\cdot_\h\varpi(u,w)+\varpi(v,u)\cdot_\h w-\varpi(u,v)\cdot_\h w  \\
\nonumber&&-\varpi(v,u\cdot_\h w)+\varpi(u,v\cdot_\h w)+\varpi([u,v]_\h,w),\\
\label{varpi bracket}0&=&\varpi(\varpi(u,v),w)-\varpi(u,\varpi(v,w))-\varpi(\varpi(v,u),w)+\varpi(v,\varpi(u,w)),\\
\label{1-cocycle phi}0&=&\phi(x)\cdot_\h\theta(y)+\theta(x)\cdot_\h\phi(y)-\theta(x\cdot_\g y)-\phi(\omega(x,y))+\varpi(\phi(x),\phi(y)),\\
\label{eq:theta2}0&=&\theta(\omega(x,y))-\varpi(\phi(x),\theta(y))-\varpi(\theta(x),\phi(y))-\theta(x)\cdot_\h \theta(y).
\end{eqnarray}
Note that  $(\ref{2-cocycle g})$, \eqref{eq:2-cocycle h} and  \eqref{1-cocycle phi} mean that $\delta_{\rm preLie}(\omega,\varpi,\theta)=0$, $(\ref{omega bracket})$ means that $(\g,\omega)$ is a pre-Lie algebra, and $(\ref{varpi bracket})$ means that $(\h,\varpi)$ is a pre-Lie algebra.

\begin{defi}
Two infinitesimal deformations $(\g_t,\h_t,\phi_t)$ and 	$(\g'_t,\h'_t,\phi'_t)$  of a pre-Lie-morphism triple $(\g,\h,\phi)$,  which are generated by $(\omega,\varpi,\theta)$ and $(\omega',\varpi',\theta')$, are said to be {\bf equivalent} if there exist a family of pre-Lie algebra isomorphisms $\Id_\g+tN:\g_t\to\g'_t$ and $\Id_\h+tS:\h_t\to\h'_t$ with $N\in\gl(\g)$ and $S\in\gl(\h)$, such that
$$\phi'_t\circ (\Id_\g+tN)=(\Id_\h+tS)\circ \phi_t.$$
A infinitesimal deformation $(\g_t,\h_t,\phi_t)$ of a pre-Lie-morphism triple $(\g,\h,\phi)$ is said to be {\bf trivial} if it is equivalent to $((\g,\cdot_\g),(\h,\cdot_\h),\phi)$.
\end{defi}

By a direct calculation, $(\g_t,\h_t,\phi_t)$ and 	$(\g'_t,\h'_t,\phi'_t)$  of a pre-Lie-morphism triple $(\g,\h,\phi)$ are equivalent if and only if
\begin{eqnarray}
\omega(x,y)-\omega'(x,y)&=&x\cdot_\g N(y)+N(x)\cdot_\g y-N(x\cdot_\g y),\label{2-exact}\\
N\omega(x,y)&=&\omega'(x,N(y))+\omega'(N(x),y)+N(x)\cdot_\g N(y),\label{integral condition 1}\\
\omega'(N(x),N(y))&=&0,\label{eq:con1}\\
\varpi(u,v)-\varpi'(u,v)&=&u\cdot_\h S(v)+S(u)\cdot_\h v-S(u\cdot_\h v),\label{2-exact h}\\
S\varpi(u,v)&=&\varpi'(u,S(v))+\varpi'(S(u),v)+S(u)\cdot_\h S(v),\label{integral condition 1 h}\\
\varpi'(S(u,S(v))&=&0,\label{eq:con1 h}\\
\label{eq:relation7}\theta-\theta'&=&\phi\circ N-S\circ \phi,\\
\label{eq:relation8}\theta'\circ N&=&S\circ \theta.
\end{eqnarray}
Note that  $(\ref{2-exact})$, \eqref{2-exact h} and  \eqref{eq:relation7} mean that $(\omega,\varpi,\theta)-(\omega',\varpi',\theta')=\delta_{\rm preLie}(N,S,0)$.

We summarize the above discussion into
the following theorem:
\begin{thm}\label{thm:deformation}
Let $(\g_t,\h_t,\phi_t)$  be a infinitesimal deformation  of a pre-Lie-morphism triple $(\g,\h,\phi)$ generated by $(\omega,\varpi,\theta)$.
Then $(\omega,\varpi,\theta)$ is closed, i.e. $\delta_{\rm preLie}(\omega,\varpi,\theta)=0$.  Furthermore, if two infinitesimal deformations
$(\g_t,\h_t,\phi_t)$ and 	$(\g'_t,\h'_t,\phi'_t)$  of the pre-Lie-morphism triple $(\g,\h,\phi)$ generated by $(\omega,\varpi,\theta)$ and $(\omega',\varpi',\theta')$ are equivalent,  then $(\omega,\varpi,\theta)$ and $(\omega',\varpi',\theta')$  are in the same cohomology class in $H^1(\phi,\phi)$.
\end{thm}

Now we consider trivial deformations of a pre-Lie-morphism triple  $(\g,\h,\phi)$. Then \eqref{2-exact}-\eqref{eq:relation8} reduce to
\begin{eqnarray}
	\omega(x,y)&=&x\cdot_\g N(y)+N(x)\cdot_\g y-N(x\cdot_\g y),\label{eq:trivial1}\\
N\omega(x,y)&=&N(x)\cdot_\g N(y),\label{eq:trivial2}\\
\varpi(u,v)&=&u\cdot_\h S(v)+S(u)\cdot_\h v-S(u\cdot_\h v),\label{eq:trivial3}\\
S\varpi(u,v)&=&S(u)\cdot_\h S(v),\label{eq:trivial4}\\
\theta&=&\phi\circ N-S\circ \phi,\label{eq:trivial5}\\
S\circ \theta&=&0.\label{eq:trivial6}
\end{eqnarray}
It follows from \eqref{eq:trivial1} and \eqref{eq:trivial2}  that $N$ must satisfy the following condition:
\begin{eqnarray}
N(x)\cdot_\g N(y)-N(x\cdot_\g N(y)+N(x)\cdot_\g y-N(x\cdot_\g
y))=0,\label{eq:Nijenhuis1}
\end{eqnarray}
which means that $N$ is a Nijenhuis operator on the pre-Lie algebra $(\g,\cdot_\g)$.

It follows from \eqref{eq:trivial3} and \eqref{eq:trivial4}  that $S$ must satisfy the following condition:
\begin{eqnarray}
S(u)\cdot_\h S(v)-S(u\cdot_\h S(v)+S(u)\cdot_\h v-S(u\cdot_\h
v))=0,\label{eq:Nijenhuis2}
\end{eqnarray}
which means that $S$ is a Nijenhuis operator on the pre-Lie algebra $(\h,\cdot_\h)$.

It follows from \eqref{eq:trivial5} and \eqref{eq:trivial6}  that $N$ and $S$ must satisfy the following condition:
\begin{eqnarray}
S\circ\phi\circ N=S^2\circ \phi.\label{eq:Nijenhuis3}
\end{eqnarray}
\begin{defi}
Let $((\g,\cdot_\g),(\h,\cdot_\h),\phi)$ be a pre-Lie-morphism triple. A pair $(N,S)$, where $N\in\gl(\g)$ and $S\in\gl(\h)$, is called a {\bf Nijenhuis pair} on the pre-Lie-morphism triple $(\g,\h,\phi)$ if $N$ is a Nijenhuis operator on the pre-Lie algebra $(\g,\cdot_\g)$,  $S$ is a Nijenhuis operator on the pre-Lie algebra $(\h,\cdot_\h)$ and  \eqref{eq:Nijenhuis3} holds.
\end{defi}

We have seen that a trivial deformation of a pre-Lie-morphism triple $(\g,\h,\phi)$ could give rise to a Nijenhuis pair. In fact, the converse is also true.

\begin{thm}\label{thm:trivial def}
  Let $(N,S)$ be a Nijenhuis pair on the 	pre-Lie-morphism triple $(\g,\h,\phi)$. Then a deformation of  the 	pre-Lie-morphism triple $\phi$ can be
  obtained by putting
\begin{eqnarray}
\label{eq:exact 1}\omega(x,y)&=& \dM_{\g}N(x,y),\\
\label{eq:exact 2}\varpi(u,v)&=& \dM_{\h}S(u,v),\\
\label{eq:exact 3}\theta(x)&=&\phi(N(x))-S(\phi(x)),\quad\forall~x,y\in\g,u,v\in\h.
\end{eqnarray}

  Furthermore, this deformation is trivial.
\end{thm}
\begin{proof}
By \eqref{eq:exact 1}-\eqref{eq:exact 3}, we have $(\omega,\varpi,\theta)=\delta_{\rm preLie}(N,S,0)$. Then $(\omega,\varpi,\theta)$ is closed naturally	, which implies that $(\ref{2-cocycle g})$, \eqref{eq:2-cocycle h} and  \eqref{1-cocycle phi} hold. $(\ref{omega bracket})$  and $(\ref{varpi bracket})$  follow from that $N$ and $S$ are Nijenhuis operators on pre-Lie algebras $\g$ and $\h$, respectively.  We will show that \eqref{eq:theta2} holds. By the fact that $\phi$ is a pre-Lie homomorphism and \eqref{eq:Nijenhuis3} , we have
\begin{eqnarray*}
&&\theta(\omega(x,y))-\varpi(\phi(x),\theta(y))-\varpi(\theta(x),\phi(y))-\theta(x)\cdot_\h \theta(y)\\
&=&S(\phi(x))\cdot_\h S(\phi(y))-S(S(\phi(x))\cdot_\h\phi(y)+\phi(x)\cdot_\h S(\phi(y))-S(\phi(x)\cdot_\h \phi(y)))	\\
&=&0.
\end{eqnarray*}
The last equality follows from $S$ is a Nijenhuis operator on the pre-Lie algebra $\g$.

At last, it is obvious that
\begin{eqnarray*}
	(\Id_\g +tN)(x\cdot_t y)&=&(\Id_\g +tN)(x)\cdot_\g (\Id_\g +tN)(y),\\
		(\Id_\h +tS)(u\ast_t v)&=&(\Id_\h +tS)(u)\cdot_\h (\Id_\h+tS)(v),\\
		\phi((\Id_\g+tN)(x))&=&(\Id_\h+tS)((\phi+t\theta)(x)),\quad \forall~x,y\in\g,u,v\in\h.
\end{eqnarray*}
Thus the deformation is trivial.
\end{proof}

\subsection{Nijenhuis pairs on pre-Lie-morphism triples and compatible structures on pre-Lie algebras}
Let $(\g,\cdot_\g)$ be a pre-Lie algebra and $(V;\rho,\mu)$ a representation. Let $T_1,T_2: V\longrightarrow \g$ be two
$\huaO$-operators on the pre-Lie algebra $\g$ associated to the representation $(V;\rho,\mu)$. Recall that $T_1$ and $T_2$ are {\bf compatible} if for any
$k_1,k_2$, $k_1T_1+k_2T_2$ is still an $\huaO$-operator.

A pair of compatible $\huaO$-operators can  give rise to a Nijenhuis operator under some conditions.
\begin{pro}{\rm(\cite{WBLS})}\label{thm:comRBN}
Let $T_1,T_2: V\longrightarrow \g$ be two $\huaO$-operators on
a pre-Lie algebra $\g$ associated to a representation $(V;\rho,\mu)$. Suppose that $T_2$ is invertible. If $T_1$ and $T_2$
are compatible, then $N=T_1\circ T_2^{-1}$ is a Nijenhuis operator on the pre-Lie algebra $(\g,\cdot_\g)$.
\end{pro}

Given two compatible $\huaO$-operators $T_1$ and $T_2$ on the pre-Lie algebra $\g$ associated to the representation $(V;\rho,\mu)$, by Example \ref{ex:O-operators}, $((V,\cdot^{T_1}),(\g,\cdot_\g),T_1)$ and $((V,\cdot^{T_2}),(\g,\cdot_\g),T_2)$ are pre-Lie-morphism triples.
\begin{lem}\label{lem:Nijenhuis pair 1}
Let $T_1,T_2: V\longrightarrow \g$ be two
$\huaO$-operators on the pre-Lie algebra $\g$ associated to the representation $(V;\rho,\mu)$. Suppose that $T_2$ is invertible. If
$T_1$ and $T_2$ are compatible, then
\begin{itemize}
  \item[\rm(i)]$N\circ T_1=T_1\circ S$ and $S$ is a Nijenhuis operator on the pre-Lie algebra $(V,\cdot^{T_1})$,  where $N=T_1\circ T_2^{-1}$ and $ S=T^{-1}_2\circ T_1$;
   \item[\rm(ii)]$N\circ T_2=T_2\circ S$ with $S=T^{-1}_2\circ T_1$  and $S$ is a Nijenhuis operator on the pre-Lie algebra $(V,\cdot^{T_2})$, where $N=T_1\circ T_2^{-1}$ and $ S=T^{-1}_2\circ T_1$.
\end{itemize}
\end{lem}
\begin{proof}
It follows from  Theorem 4.4 and Proposition 4.21 in \cite{LW}.
\end{proof}

\begin{pro}\label{pro:Nij-pair-o}{
Let $T_1,T_2: V\longrightarrow \g$ be two
$\huaO$-operators on the pre-Lie algebra $\g$ associated to the representation $(V;\rho,\mu)$. Suppose that $T_2$ is invertible. If
$T_1$ and $T_2$ are compatible, then $(N,S)$ is both a Nijenhuis pair on the pre-Lie-morphism triples $((V,\cdot_{T_1}),(\g,\cdot_\g),T_1)$ and $((V,\cdot_{T_2}),(\g,\cdot_\g),T_2)$, where $N=T_1\circ T_2^{-1}$ and $ S=T^{-1}_2\circ T_1$.

Furthermore, $((V,\cdot^{T_1}_S),(\g,\cdot_N),T_1)$ and $((V,\cdot^{T_2}_S),(\g,\cdot_N),T_2)$ are pre-Lie-morphism triples, where the multiplication $\cdot^{T_i}_S:\otimes^2V\longrightarrow V$ for $i=1,2$ is given by
\begin{eqnarray}
 \label{eq:defieq21}u\cdot^{T_i}_S v&=&S(u)\cdot^{T_i} v+u\cdot^{T_i} S(v)-S(u\cdot^{T_i} v),\quad\forall~u,v\in V.
\end{eqnarray}

}
\end{pro}
\begin{proof}
 By Lemma \ref{lem:Nijenhuis pair 1}, $(N,S)$ is the Nijenhuis pair on the pre-Lie-morphism triples $((V,\cdot_{T_1}),(\g,\cdot_\g),T_1)$ and $((V,\cdot_{T_2}),(\g,\cdot_\g),T_2)$. By the fact that $T_1$ is an $\huaO$-operators on the pre-Lie algebra $\g$, we have $T_1(u\cdot^{T_1}v)=T_1(u)\cdot_\g T_1(v)$ for all $u,v\in\g$. Furthermore, by $N\circ T_1=T_1\circ S$, we have
 \begin{eqnarray*}
   T_1(u\cdot^{T_1}_S v )&=&T_1(S(u)\cdot^{T_1}v+u\cdot^{T_1}S(v)-S(u\cdot^{T_1}v))\\
   &=&N(T_1(u))\cdot_\g T_1(v)+T_1(u)\cdot^{T_1}N(T_1(v))-N(T_1(u)\cdot_\g T_1(v))\\
   &=&T_1(u)\cdot_N T_1(v).
 \end{eqnarray*}
 Thus $((V,\cdot^{T_1}_S),(\g,\cdot_N),T_1)$ is a pre-Lie-morphism triple. Similarly, $((V,\cdot^{T_2}_S),(\g,\cdot_N),T_2)$ is also a pre-Lie-morphism triple.
\end{proof}

By Example \ref{ex:s-matric} and Proposition \ref{pro:Nij-pair-o}, we have
\begin{cor}
  Let $r_1,r_2\in\Sym^2(\g)$ be two $\frks$-matrices on the pre-Lie algebra $\g$. Assume that $r_2$ is invertible. If $r_1$ and $r_2$ are compatible in the sense that any linear combination of $r_1$ and $  r_2$ is still an $\frks$-matrix, then $(N=r_1^\sharp\circ (r_2^\sharp)^{-1},S=(r_2^\sharp)^{-1}\circ r_1^\sharp)$ is both a Nijenhuis pair on the pre-Lie-morphism triples $((\g^*,\cdot^{r_1}),(\g,\cdot_\g),r_1^\sharp)$ and $((\g^*,\cdot^{r_2}),(\g,\cdot_\g),r_2^\sharp)$, where the operation $\cdot^{r_i}$ for $i=1,2$ is given by \eqref{eq:pre-bia}.
\end{cor}

\section{Cohomologies of pre-Lie-morphism triples and  Lie-morphism triples}\label{sec:relation}
In this section, we show that the cohomology for a pre-Lie-morphism triple can be deduced from a new cohomology of the sub-adjacent Lie-morphism triple.

The Chevalley-Eilenberg  cohomology theory for a Lie algebra $(\g, [-,-]_\g)$ associated to a representation $(V;\rho)$ is given as follows. Denote by $C^n_{\CE}(\g,V):=\Hom(\wedge^n \g,V)$, the space of $n$-cochains. The corresponding Chevalley-Eilenberg coboundary operator
$\partial :C_{\CE}^n(\g,V)\longrightarrow C_{\CE}^{n+1}(\g,V)$ is given by
  \begin{eqnarray}
\label{eq:coboundary of Lie}\partial f(x_1,\ldots,x_{n+1})&=&\sum_{i=1}^{n+1}(-1)^{i+1}\rho({x_i}) f(x_1,\ldots,\hat{x_i},\ldots,x_{n+1})\\
\nonumber&&+\sum_{1\leq i<j\leq n+1}^{n+1}(-1)^{i+j}f([x_i,x_j]_\g,\ldots,\hat{x_{i}},\ldots,\hat{x_{j}},\ldots,x_{n+1})
\end{eqnarray}
for all $f\in C_{\CE}^n(\g,V)$ and $x_1,x_2,\cdots,x_{n+1}\in \g$. We denote the corresponding $n$-th cohomology group by $H_{\rm CE}^n(\g,V)$.

Let $(\g,\cdot_\g)$ be a pre-Lie algebra and $(V;\rho,\mu)$ a representation. Define $\varrho:\g\to\gl(\Hom(\g,V))$ by
\begin{equation}
\varrho(x)(f)(y)=\rho(x)f(y)-\mu(y)f(x)-f(x\cdot_\g y),\quad\forall~f\in\Hom(\g,V),x,y\in\g.
\end{equation}
Then it was shown in \cite{cohomology of pre-Lie} that $\varrho$ is a representation of the sub-adjacent Lie algebra $\g^c$ on the vector space $\Hom(\g,V)$. Furthermore, we define $\Phi:C^{k}_{\CE}(\g,\Hom(\g,V))\to 	C_{\rm preLie}^{k+1}(\g,V)$ by
\begin{equation}
  \Phi(f)(x_1,\cdots,x_{n},y)=f(x_1,\cdots,x_{n})(y),\quad\forall~x_1,\cdots,x_{n},y\in \g.
\end{equation}
\begin{thm}{\rm (\cite{cohomology of pre-Lie,Nijenhuis})}\label{thm:isomorphism}
  The map $\Phi$ is a cochain map from the cochain complex $(C^{*}_{\CE}(\g,\Hom(\g,V)),\partial_\varrho)$ to the cochain complex $(C_{\rm preLie}^{*}(\g,V),\dM)$, where $\partial_\varrho$ is the coboundary operator of the sub-adjacent Lie algebra $\g^c$  associated to   representation $\varrho$ and $\dM$ is the coboundary operator of the pre-Lie algebra $\g$  associated to  representation $(\rho,\mu)$.

  Moreover, $\Phi$ induces an isomorphism between the corresponding cohomology groups:
  \begin{equation}
H^{n}_{\rm preLie}(\g,V)\cong H^{n-1}_{\rm CE}(\g,\Hom(\g,V)),\quad n=1,2\cdots.
  \end{equation}
\end{thm}

Let $\phi$ be a homomorphism between pre-Lie algebras $(\g,\cdot_\g)$ and $(\h,\cdot_\h)$. Define $\varrho_\g:\g\to \gl(\Hom(\g,\g))$,  $\varrho_\h:\h\to \gl(\Hom(\h,\h))$ and $\varrho_\phi:\g\to \gl(\Hom(\g,\h))$ by
\begin{eqnarray}
\varrho_\g(x)(f_1)(y)&=&x\cdot_\g f_1(y)-f_1(x)\cdot_\g y-f_1(x\cdot_\g y),\\
\varrho_\h(u)(f_2)(v)&=&u\cdot_\h f_2(v)-f_2(u)\cdot_\h v-f_2(u\cdot_\h v),\\
\varrho_\phi(x)(f_3)(y)&=&\rho_\phi(x)f_3(y)-\mu_\phi(y)f_3(x)-f_3(x\cdot_\g y),
\end{eqnarray}
where $f_1\in\Hom(\g,\g)$, $f_2\in\Hom(\h,\h)$, $f_3\in\Hom(\g,\h)$ and $(\rho_\phi,\mu_\phi)$ is given by  \eqref{eq:rep}. Then $\varrho_\g$,  $\varrho_\h$ and $\varrho_\phi$ are presentations of the corresponding  sub-adjacent Lie algebras.

In the following, we generalize Theorem \ref{thm:isomorphism} to the case of pre-Lie-morphism triples. Define the set of $k$-cochains by
$$\huaC_{\CE}^k(\phi,\phi)=C_{\CE}^{k+1}(\g,\Hom(\g,\g))\oplus C_{\CE}^{k+1}(\h,\Hom(\h,\h))\oplus C_{\CE}^{k}(\g,\Hom(\g,\h)),k\geq0,$$
in which we set $C_{\CE}^{0}(\g,\Hom(\g,\h))=0$.

The coboundary operator $\delta_{\CE}:\huaC_{\CE}^k(\phi,\phi)\to \huaC_{\CE}^{k+1}(\phi,\phi)$ is defined  by
$$\delta_{\CE}(f_1,f_2,f_3)=(\partial_\g f_1,\partial_\h f_2,\partial_\phi f_3+(-1)^k(\phi\circ f_1-\phi^*\circ f_2)),$$
for all  $f_1\in C_{\CE}^{k+1}(\g,\Hom(\g,\g)),f_2\in C_{\CE}^{k+1}(\h,\Hom(\h,\h)),f_3\in C_{\CE}^{k}(\g,\Hom(\g,\h))$, where $\partial_\g$, $\partial_\h$ and $\partial_\phi$ are Chevalley-Eilenberg coboundary operators of the sub-adjacent Lie algebras $\g^c$, $\h^c$ and $\g^c$ associated to  representations $\varrho_\g$, $\varrho_\h$  and $\varrho_\phi$, respectively. Here $\phi\circ f_1\in C_{\CE}^{k+1}(\g,\Hom(\g,\h))$ and $\phi^*\circ f_2\in C_{\CE}^{k+1}(\g,\Hom(\g,\h))$ are defined by
\begin{eqnarray*}
(\phi\circ f_1)(x_1,\cdots,x_{k+1})(y)&=&\phi\big( f_1(x_1,\cdots,x_{k+1})(y)\big),\\
(\phi^*\circ f_2)(x_1,\cdots,x_{k+1})(y)&=&f_2 (\phi(x_1),\cdots,\phi(x_{k+1}))(\phi(y)),
\end{eqnarray*}
where $x_1,\cdots,x_{k+1},y\in\g$.

\begin{pro}
With the above notations,	we have $\delta_{\CE}\circ\delta_{\CE}=0$. Hence $(\huaC_{\CE}^*(\phi,\phi)=\otimes_{k\geq 1}C^n(\phi,\phi),\delta_{\CE})$ is a cochain complex. We denote the corresponding $k$-th cohomology group  by $H_{\CE}^k(\phi,\phi)$.
\end{pro}	
\begin{proof}
For $f_1\in C_{\CE}^{k}(\g,\Hom(\g,\g))$, we have
\begin{eqnarray*}
\label{eq:pre-Lie cohomology}
&&\partial_\phi(\phi\circ f_1)(x_1,\cdots,x_{k+1})(y)\\
\nonumber&=&\sum_{i=1}^{k+1}(-1)^{i+1}\phi(x_i)\cdot_\h \phi(f_1(x_1, \cdots,\hat{x_i},\cdots,x_{k+1})(y))\\
\nonumber &&\sum_{i=1}^{k+1}(-1)^{i+1}\phi(f_1(x_1, \cdots,\hat{x_i},\cdots,x_{k+1})(x_i))\cdot_\h\phi(y)\\
\nonumber&&-\sum_{i=1}^{k+1}(-1)^{i+1}\phi(f_1(x_1, \cdots,\hat{x_i},\cdots,x_{k+1})(x_i\cdot_\g y))\\
\label{eq:cobold} &&+\sum_{1\leq i<j\leq {k+1}}(-1)^{i+j}\phi(f_1([x_i,x_j]_\g,x_1,\cdots,\hat{x_i},\cdots,\hat{x_j},\cdots,x_{k+1})(y))\\
\nonumber&=&\phi(\sum_{i=1}^{k+1}(-1)^{i+1} x_i\cdot_\g f_1(x_1, \cdots,\hat{x_i},\cdots,x_{k+1})(y))\\
\nonumber &&-\phi(\sum_{i=1}^{k+1}(-1)^{i+1} f_1(x_1, \cdots,\hat{x_i},\cdots,x_{k+1})(x_i)\cdot_\g y)\\
\nonumber&&-\phi(f_1([x_i,x_j]_\g,x_1,\cdots,\hat{x_i},\cdots,\hat{x_j},\cdots,x_{k+1})(x_i\cdot_\g y))\\
\label{eq:cobold} &&+\phi(\sum_{1\leq i<j\leq {k+1}}(-1)^{i+j}f_1([x_i,x_j]_\g,x_1,\cdots,\hat{x_i},\cdots,\hat{x_j},\cdots,x_{k+1})(y))\\
&=&\phi(\partial_\g(f_1)(x_1, \cdots,x_{k+1})(y)),\\
\nonumber&=&(\phi\circ \partial_\g(f_1))(x_1,\cdots,x_{k+1})(y)
\end{eqnarray*}
which implies that
\begin{eqnarray}\label{eq:bre1}
\partial_\phi(\phi\circ f_1)=\phi\circ \partial_\g(f_1).
\end{eqnarray}
Similarly, for $f_2\in C_{\CE}^{k+1}(\h,\Hom(\h,\h))$, we have
\emptycomment{we have
\begin{eqnarray*}
\label{eq:pre-Lie cohomology}
\partial_\phi(\phi^* f_2) (x_1, \cdots,x_{k+1})(y)
\nonumber&=&\sum_{i=1}^{k+1}(-1)^{i+1}\phi(x_i)\cdot_\h f_2(\phi(x_1, \cdots,\hat{x_i},\cdots,x_{k+1})(\phi(y)))\\
\nonumber &&-\sum_{i=1}^{k+1}(-1)^{i+1}\phi(x_i)\cdot_\h f_2(x_1, \cdots,\hat{x_i},\cdots,x_{k+1})(\phi(y))\\
\nonumber&&-\sum_{i=1}^{k+1}(-1)^{i+1}f_2(\phi(x_i\cdot_\g x_1, \cdots,\hat{x_i},\cdots, x_{k+1}(\phi(y))))\\
\label{eq:cobold} &&+\sum_{1\leq i<j\leq {k+1}}(-1)^{i+j}f_2(\phi([x_i,x_j]_\g,x_1,\cdots,\hat{x_i},\cdots,\hat{x_j},\cdots,x_{k+1})(\phi(y)))\\
\nonumber&=&\sum_{i=1}^{k+1}(-1)^{i+1}\phi(x_i)\cdot_\h f_2(\phi(x_1, \cdots,\hat{x_i},\cdots,x_{k+1})(\phi(y)))\\
\nonumber &&-\sum_{i=1}^{k+1}(-1)^{i+1}f_2(\phi(x_i))\cdot_\h \phi(x_1, \cdots,\hat{x_i},\cdots,x_{k+1})(\phi(y))\\
\nonumber&&-\sum_{i=1}^{k+1}(-1)^{i+1}f_2(\phi(x_i)\cdot_\h \phi(x_1), \cdots,\hat{x_i},\cdots, \phi(x_{k+1})(\phi(y)))\\
\label{eq:cobold} &&+\sum_{1\leq i<j\leq {k+1}}(-1)^{i+j}f_2([\phi(x_i),\phi(x_j)]_\h,\phi(x_1),\cdots,\hat{x_i},\cdots,\hat{x_j},\cdots,\phi(x_{k+1})(\phi(y)))\\
\nonumber&=&\phi^*\partial_\h(f_2)(x_1, \cdots,x_{k+1})(y),
\end{eqnarray*}
Which implies that}
\begin{eqnarray}\label{eq:bre2}
\partial_\phi(\phi^*\circ  f_2)=\phi^*\circ \partial_\h(f_2).
\end{eqnarray}

Furthermore, by \eqref{eq:bre1}, \eqref{eq:bre2} and the fact that $\partial_\g\circ \partial_\g=0$ , $\partial_\h\circ \partial_\h=0$ and $\partial_\phi\circ \partial_\phi=0$, we have
\begin{eqnarray*}
&&\delta_{\CE}\circ\delta_{\CE}(f_1,f_2,f_3)\\
&=&\delta_{\CE}(\partial_\g f_1,\partial_\h f_2,\partial_\phi f_3+(-1)^k(\phi\circ f_1-\phi^*\circ f_2))\\
&=&(\partial_\g\circ\partial_\g f_1,\partial_\h\circ\partial_\h f_2,\partial_\phi\circ(\partial_\phi f_3+(-1)^k(\phi\circ f_1-\phi^*\circ f_2))+(-1)^{k+1}(\phi\circ\partial_\g f_1-\phi^*\circ\partial_\h f_2))\\
&=&\Big(0,0,(-1)^k\big(\partial_\phi(\phi\circ f_1)-\phi\circ \partial_\g(f_1)+\phi^*\circ \partial_\h(f_2)-\partial_\phi(\phi^* \circ f_2)\big)\Big)=(0,0,0),
\end{eqnarray*}
which implies that $\delta_{\CE}\circ\delta_{\CE}=0$.
\end{proof}	
\begin{rmk}
  In \cite{Das}, the author gives a cohomology of a Lie-morphism triple with coefficients in a representation. Our cohomology for the sub-adjacent Lie-morphism triple in this paper is not fit in the mentioned above cohomology theory.
\end{rmk}

Furthermore, we define $\Phi:C^{k}_{\CE}(\phi,\phi)\to 	C_{\rm preLie}^{k+1}(\phi,\phi)$ by
\begin{equation}
\Phi(f_1,f_2,f_3)=(\Phi_1(f_1),\Phi_2(f_2),\Phi_3(f_3)),
\end{equation}
where $\Phi_1:C_{\CE}^{k+1}(\g,\Hom(\g,\g))\to C_{\rm preLie}^{k+2}(\g,\g)$, $\Phi_2:C_{\CE}^{k+1}(\h,\Hom(\h,\h))\to C_{\rm preLie}^{k+2}(\h,\h)$ and $\Phi_3:C_{\CE}^{k}(\g,\Hom(\g,\h))\to C_{\rm preLie}^{k+1}(\g,\h)$ are given by
\begin{eqnarray*}
	  \Phi_1(f_1)(x_1,\cdots,x_{k+1},y)&=&f_1(x_1,\cdots,x_{k+1})(y),\\
	    \Phi_2(f_2)(u_1,\cdots,u_{k+1},v)&=&f_2(u_1,\cdots,u_{k+1})(v),\\
	      \Phi_3(f_3)(x_1,\cdots,x_{k},y)&=&f_3(x_1,\cdots,x_{k})(y)
\end{eqnarray*}
for all $x_1,\cdots,x_{k+1},y\in\g$, $u_1,\cdots,u_{k+1},v\in\h$.

\begin{thm}
The map $\Phi$ is an isomorphic cochain map from the cochain complex $(C^{*}_{\CE}(\phi,\phi),\delta_{\CE})$ to the cochain complex $(C_{\rm preLie}^{*}(\phi,\phi),\delta_{\rm preLie})$.

  Moreover, $\Phi$ induces an isomorphism between the corresponding cohomology groups:
  \begin{equation}
H^{k}_{\rm preLie}(\phi,\phi)\cong H^{k-1}_{\rm CE}(\phi,\phi),\quad k=1,2,\cdots.
  \end{equation}
\end{thm}
\begin{proof}
It is obvious that $\Phi$ is an isomorphic map. In the following, we need to show that  $\delta_{\CE}\circ\Phi=\Phi\circ \delta_{\rm preLie}$. By a direct calculation, we have
\begin{eqnarray*}
   \delta_{\CE}\circ\Phi(f_1,f_2,f_3)&=&(\partial_\g\Phi_1(f_1),\partial_\h\Phi_2(f_2),\partial_\phi\Phi_3(f_3)+(-1)^k(\phi\circ\Phi_1(f_1)-\phi^*\circ\Phi_2 (f_2)));\\
\Phi\circ \delta_{\rm preLie}(f_1,f_2,f_3)
&=&(\Phi_1(\dM_\g f_1),\Phi_2(\dM_\h f_2),\Phi_3(\dM_\phi f_3)+(-1)^k\Phi_3(\phi\circ f_1-\phi^*\circ f_2)).
\end{eqnarray*}
By Theorem \ref{thm:isomorphism}, we have
\begin{eqnarray*}
\partial_\g\Phi_1(f_1)=\Phi_1(\dM_\g f_1),\quad\partial_\h\Phi_2(f_2)=\Phi_2(\dM_\h f_2),\quad \partial_\phi\Phi_3(f_3)=\Phi_3(\dM_\phi f_3).
\end{eqnarray*}
Thus $\delta_{\CE}\circ\Phi=\Phi\circ \delta_{\rm preLie}$ holds if and  only if $$\phi\circ\Phi_1(f_1)-\phi^*\circ\Phi_2 (f_2)=\Phi_3(\phi\circ f_1-\phi^*\circ f_2).$$	
By a direct calculation, we have
\begin{eqnarray*}
\phi\circ\Phi_1(f_1)(x_1,\cdots,x_{k+1},y)&=&\phi(f_1(x_1,\cdots,x_{k+1})(y)),\\
&=&(\phi\circ f_1)(x_1,\cdots,x_{k+1})(y)\\
&=&\Phi_3(\phi\circ f_1)(x_1,\cdots,x_{k+1},y),
\end{eqnarray*}
which implies that
 $$\phi\circ\Phi_1(f_1)=\Phi_3(\phi\circ f_1).$$
Similarly, we also have
 $$\phi^*\circ\Phi_2 (f_2)=\Phi_3(\phi^*\circ f_2).$$
Thus $\delta_{\CE}\circ\Phi=\Phi\circ \delta_{\rm preLie}$ holds,  i.e. the map $\Phi$ is an isomorphic cochain map from the cochain complex $(C^{*}_{\CE}(\phi,\phi),\delta_{\CE})$ to the cochain complex $(C_{\rm preLie}^{*}(\phi,\phi),\delta_{\rm preLie})$.

The second conclusion follows directly.
\end{proof}

\vspace{3mm}
\noindent
{\bf Acknowledgements. } This research is supported by NSFC (11901501).

 \end{document}